\documentclass[reqno]{amsart}
\usepackage{amsmath,amssymb,euscript,graphicx}
\usepackage[arrow,matrix,curve]{xy}
\numberwithin{equation}{section}

\newcommand\bR{\mathbb{R}}
\newcommand\cF{\EuScript{F}}
\newcommand\cR{\EuScript{R}}

\newcommand\id{\operatorname{id}}
\newcommand\im{\operatorname{im}}
\newcommand\dr{\operatorname{d}}

\newcommand\pdr\partial
\newcommand\J{\mathrm{J}}
\newcommand\F{\mathrm{F}}
\newcommand\Gr{\mathrm{Gr}}

\newcommand{\Span}{\mathrm{span}}

\newcommand\al\alpha
\newcommand\gam\gamma
\newcommand\del\delta
\newcommand\veps\varepsilon
\newcommand\lam\lambda
\newcommand\tht\theta
\newcommand\om\omega
\newcommand\Del\varDelta
\newcommand\Gam\varGamma
\newcommand\Om\varOmega

\newcommand\map[1]{\stackrel{{#1}}\longrightarrow}

\newcommand\mfrac[2]{\raisebox{-.1ex}{\scalebox{1.3}[1.3]{$\frac{#1}{#2}$}}}
\newcommand\msum[2]{\raisebox{.1ex}{\scalebox{0.8}[0.8]
                   {$\displaystyle\sum_{#1}^{#2}$}}}
\newcommand\moplus[2]{\raisebox{.15ex}{\scalebox{0.9}[0.9]
                   {$\displaystyle\bigoplus_{#1}^{#2}$}}}
\newcommand\mcup[2]{\raisebox{.15ex}{\scalebox{0.9}[0.9]
                   {$\displaystyle\bigcup_{#1}^{#2}$}}}

\theoremstyle{definition}
\newtheorem{definition}{Definition}[section]

\theoremstyle{plain}
\newtheorem{theorem}[definition]{Theorem}
\newtheorem{proposition}[definition]{Proposition}
\newtheorem{lemma}[definition]{Lemma}
\newtheorem{corollary}[definition]{Corollary}

\begin{document}
\setlength{\abovedisplayskip}{3pt}
\setlength{\belowdisplayskip}{3pt}

\title{On a new filtration of the variational bicomplex}
\maketitle

\begin{center}
{\large Siye Wu}\footnote{E-mail address: {\tt swu@math.nthu.edu.tw}}\\
\medskip
{\small Department of Mathematics, National Tsing Hua University,
Hsinchu 30013, Taiwan}\\
\medskip
and\\
\medskip
{\large Haoran Yang}\footnote{E-mail address: {\tt Haoran.Yang@anu.edu.au}}\\
\medskip
{\small Mathematical Sciences Institute, The Australian National University,
Canberra, ACT 2600, Australia}
\end{center}

\begin{abstract}
We define a filtration on the variational bicomplex according to jet order.
The filtration is preserved by the interior Euler operator, which is not a
module homomorphism with respect to the ring of smooth functions on the jet
space.
However, the induced maps on the graded components of this filtration are.
Furthermore, the space of functional forms in the image of the interior Euler
operator inherits a filtration.
Though the filtered subspaces are not submodules either, the graded components
are isomorphic to linear spaces which do have module structures.
This works for any fixed degree of the functional forms.
In this way, the condition that a functional form vanishes can be stated
concisely with a module basis.
We work out explicitly two examples: one for functional forms of degree two
in relation to the Helmholtz conditions and the other of arbitrary degree but
with jet order one.
\end{abstract}

\medskip

\section{Introduction}

The variational bicomplex is a double complex of differential forms on the
infinite jet bundle of a fibred manifold, designed to give a unified geometric
language for Lagrangians, variations, equations of motion, and conservation
laws.
After the pioneering work of Vinogradov \cite{Vi} and Tulczyjew \cite{Tu},
much progress has been made (see for example \cite{An,BGG,Ol,Ta,Ts}), laying
the mathematical foundation for calculus of variations and classical field
theory.
Applications of this approach cover the study of equations of motion, the
inverse problem (the Helmholtz conditions) as well as symmetries and
conservation laws (Noether's theorem).
More recently, the theory has been applied to modern physics settings such as
gauge theory, asymptotic symmetries, and BRST and BV formalisms \cite{HT,BB}.

The variational bicomplex comes from the space of differential forms on the
total space of the jet bundle.
There is a bigrading according to the horizontal and vertical degrees.
Vinogradov's filtration is according to the vertical degree, which leads to
the $C$-spectral sequence \cite{Vi}.
Instead, we define a new filtration on the variational bicomplex according
to jet order.
The filtration is preserved by the interior Euler operator.
Though this operator is not a module homomorphism with respect to the ring of
smooth functions on the jet space, the induced maps on the graded components
of this filtration are.
Furthermore, the space of functional forms in the image of the interior Euler
operator inherits a filtration as well.
Though the filtered subspaces are not submodules either, the graded components
are isomorphic to linear spaces which do have module structures.
In this way, the condition that a functional form vanishes can be stated
concisely using a module basis.
For functional forms of degree two, this condition can be worked out without
consideration of the module structure.
However, progress becomes hard when the degree is three or higher \cite{An}.
Our result for any fixed degree of the functional forms sheds light on this
complicated open problem.
We work out explicitly two examples: one for functional forms of degree two
in relation to the Helmholtz conditions and the other of arbitrary degree but
with jet order one.

In this paper, we follow the notations of Anderson \cite{An}.
The rest of the paper is organised as follows.

In Section~2, we recall the notions of (infinite) jet bundle
$\J^\infty E\to M$ of a fibre bundle $E\to M$, the ring $\cR^\infty$ of smooth
functions on the total space $\J^\infty E$ and the space of differential forms
$\Om^\bullet(\J^\infty E)$, which is an $\cR^\infty$-module with a bigrading
$\Om^{r,s}$ according to the horizontal degree $r$ and vertical degree $s$.
For classical field theory, $M$ is the spacetime manifold of dimension $n$,
and the fields are sections of the bundle $E$.
We introduce a new filtration $\F^\bullet$ on $\Om^{n,s}$ (and a subspace
$\Om^{n,s}_0$ in it) according to the jet order.
Here the filtered subspaces $\F^\bullet\Om^{n,s}$ are $\cR^\infty$-modules,
and so are the graded components $\Gr^\bullet\Om^{n,s}$.

In Section~3, motivated by calculus of variation, we review the Euler
derivative and its higher versions.
We then explain the interior Euler operator $I$, which is a projection
operator on each $\Om^{r,s}$ but fails to be $\cR^\infty$-linear.
A key observation is that the map $\bar I$ induced by $I$ on the graded
components $\Gr^\bullet\Om^{n,s}$ are nevertheless $\cR^\infty$-module
homomorphisms.
The functional forms are elements of $\cF^s:=I(\Om^{r,s})$, which is an
$\bR$-linear subspace but not an $\cR^\infty$-submodule of $\Om^{r,s}$ except
for $s=1$.
We also recall the Euler-Lagrange complex, which includes the coboundary
operator $\del_V\colon\cF^s\to\cF^{s+1}$ for each $s\ge1$.

In Section~4, we define a filtration of $\cF^s$ for each $s\ge1$ induced from
that of $\Om^{n,s}$.
Our main result is that although the filtered subspaces $\F^\bullet\,\cF^s$
in $\cF^s$ are not $\cR^\infty$-modules, the graded components
$\Gr^\bullet\,\cF^s$ are isomorphic to subspaces of $\Gr^\bullet\Om^{n,s}$
which do have $\cR^\infty$-module structures.
In fact, such a subspace is the image of the induced map $\bar I$ on the graded
components $\Gr^\bullet\Om^{n,s}$.
An element of $\cF^s$ vanishes if and only if all its graded components vanish,
a condition that is easier to handle due to the $\cR^\infty$-module structures.
These results are new, especially when $s\ge3$.

In Section~5, we specialise to $s=2$ and recover the known results, with an
efficient derivation of the Helmholtz condition using the $\cR^\infty$-module
structure.
We then work out explicitly another case of arbitrary $s\ge1$ and jet
order~$1$.

We conclude with outlook in Section~6.

\section{A filtration on the variational bicomplex according to jet order}

We begin with some preliminaries on jet spaces and functions and differential
forms on them.

Consider a classical field theory on the spacetime manifold $M$ in which fields
are bosonic and are sections of a smooth fibre bundle $\pi:E\to M$.
Let $M$ be $n$ dimensional and let $\{x^i\}_{1\le i\le n}$ be local
coordinates on $M$.
The fibres of $E$ are also finite dimensional and let $\{u^\al\}$ be their
local coordinates.
Given a multi-index $I=(i_1,\dots,i_k)$, $1\le i_1,\dots,i_k\le n$, we write
$|I|=k$ and $\pdr_I=\pdr^k/\pdr x^{i_1}\cdots\pdr x^{i_k}$.
The $k$-jet space $\J^kE$ is the total space a fibre bundle over $M$ whose
fibre at $x\in M$ consists of equivalence classes of local sections $\phi$ of
$E$ at $x$ with the same $\pdr_I\phi(x)$ for all $|I|\le k$.
Local coordinates along the fibres of $\pi^k\colon \J^kE\to M$ are $u^\al_I$,
$0\le|I|\le k$.
There are natural projections $\pi^{k+1}_k\colon\J^{k+1}E\to\J^kE$ which are
vector bundles.
We identify $\J^0E$ with $E$ and $\pi^k_0\colon\J^kE\to E$ are also vector
bundles for all $k\ge0$.
The infinite jet space $\J^\infty E$ is the inverse limit of the system of
projections $\{\pi^l_k\}_{l\ge k\ge0}$.
We then have the obvious natural projections
$\pi^\infty_k\colon\J^\infty E\to\J^kE$ for all $k\ge0$, and
$\pi^\infty\colon\J^\infty E\to M$.

Denote by $\cR^k=C^\infty(\J^kE)$ the ring of smooth functions on $\J^k(E)$.
Since $\pi^{k+1}_k$ for each $k\ge0$ is a surjection, the pull-back map
$(\pi^{k+1}_k)^*\colon\cR^k\to\cR^{k+1}$ is an inclusion, by which we can
identify $\cR^k$ as a subring of $\cR^{k+1}$.
Using this identification, we define the ring
$\cR^\infty=C^\infty(\J^\infty E)$ of smooth functions on $\J^\infty(E)$ as
the infinite union
\[ \cR^\infty:=\mcup{k\ge0}{}\,\cR^k. \]
An element of $\cR^\infty$ is a function depending on jets up to some finite
but undetermined order.
Similarly, since the pull-back map
$(\pi^{k+1}_k)^*\colon\Om^p(\J^kE)\to\Om^p(\J^{k+1}E)$ on $p$-forms is also
an inclusion for each $k\ge0$, we can regard $\Om^p(\J^kE)$ as a subspace of
$\Om^p(\J^{k+1}E)$ and define the space of $p$-forms on the infinite jet space
$\J^\infty E$ again as the infinite union
\[ \Om^p(\J^\infty E):=\mcup{k\ge0}{}\,\Om^p(\J^kE). \]
Clearly $R^\infty$ is identified as $\Om^0(\J^\infty(E))$ and
$\Om^p(\J^\infty E)$ is an $\cR^\infty$-module.

The variational bicomplex provides a canonical splitting
\[ \Om^p(\J^\infty E)=\moplus{r+s=p}{}\,\Om^{r,s}(\J^\infty E), \]
where $r$ counts horizontal degree (along $M$) of forms and $s$ counts the
vertical degree.
Vertical $1$-forms are contact forms locally given by
$\tht_I^\al=\dr\!u^\alpha_I+u^\alpha_{Ij}\dr\!x^j$, and
$\tht^{\al_1}_{I_1}\wedge\cdots\wedge\tht^{\al_s}_{I_s}$ has vertical degree
$s$.
Each $\Om^{r,s}(\J^\infty E)$ is also an $\cR^\infty$-module and the above
direct sum is a decomposition of $\cR^\infty$-modules.
Locally, an element in $\Om^{r,s}(\J^\infty E)$ is a finite $\cR^\infty$-linear
combination of $\tht^{\al_1}_{I_1}\wedge\cdots\wedge\tht^{\al_s}_{I_s}\wedge
\dr\!x^{i_1}\wedge\cdots\wedge\dr\!x^{i_r}$.
In particular, for top horizontal forms, a local $\cR^\infty$-module basis of
$\Om^{n,s}(\J^\infty E)$ is given by
$\tht^{\al_1}_{I_1}\wedge\cdots\wedge\tht^{\al_s}_{I_s}\wedge\nu$
up to the obvious antisymmetry.
Here $\nu$ is a volume form on $M$, which we assume exists and is chosen.

We now introduce an ascending filtration on the space of $(n,s)$-forms on
$\J^\infty(E)$ for each $s\ge0$.
That is, we construct filtered subspaces
$\F^l\Om^{n,s}(\J^\infty E)\subset\Om^{n,s}(\J^\infty E)$, or
$\F^l\Om^{n,s}\subset\Om^{n,s}$ for short, such that
\[ 0=\F^{-1}\Om^{n,s}\subset\F^0\Om^{n,s}\subset\F^1\Om^{n,s}\subset
   \cdots\subset\F^\infty\Om^{n,s}=\Om^{n,s}.  \]
The filtration is based on jet-order data.
For each $l\ge 0$, we define
\[ \F^l\Om^{n,s}(\J^\infty E):=\Span_{\cR^\infty}\big\{\,
   \tht^{\al_1}_{I_1}\wedge\cdots\wedge\tht^{\al_s}_{I_s}\wedge\nu
   :|I_1|+\cdots+|I_s|\le l\,\big\}. \]
Then the associated graded components are
\[ \Gr^l\Om^{n,s}:=\mfrac{\F^l\Om^{n,s}}{\F^{l-1}\Om^{n,s}}
 =\Span_{\cR^\infty}\bigl\{\,
 \tht^{\al_1}_{I_1}\wedge\cdots\wedge\tht^{\al_s}_{I_s}\wedge\nu
 :|I_1|+\cdots+|I_s|=l\,\bigr\}. \]

We remark here that the filtered spaces $\F^l\Om^{n,s}$ and their graded
components $\Gr^l\Om^{n,s}$ do not depend on the choice of local coordinates
on $M$ and $E$.
The $\bR$-vector subspaces $\F^l\Om^{n,s}(\J^\infty E)$ are actually
$\cR^\infty$-submodules of $\Om^{n,s}(\J^\infty E)$ and thus the graded
components are $\cR^\infty$-modules.
In fact, $\tht^{\al_1}_{I_1}\wedge\cdots\wedge\tht^{\al_s}_{I_s}\wedge\nu$
form an $\cR^\infty$-module basis of $\F^l\Om^{n,s}$ for
$|I_1|+\cdots+|I_s|\le l$ and an $\cR^\infty$-module basis of $\Gr^l\Om^{n,s}$
for $|I_1|+\cdots+|I_s|=l$.
We also note in passing that there is an opposite, descending filtration given
by
\[ {}'\F^l\Om^{n,s}(\J^\infty E):=\Span_{\cR^\infty}\bigl\{\,
   \tht^{\al_1}_{I_1}\wedge\cdots\wedge\tht^{\al_s}_{I_s}\wedge\nu
   :|I_1|+\cdots+|I_s|\ge l\,\bigr\}, \]
which has the same graded components.
But we will not pursue it in our present work.

In what follows we will also need an $\cR^\infty$-submodule of
$\Om^{n,s}(\J^\infty E)$ defined by
\[ \Om^{n,s}_0(\J^\infty E):=\Span_{\cR^\infty}\big\{\tht^{\al_1}\wedge
 \tht^{\al_2}_{I_2}\wedge\cdots\wedge\tht^{\al_s}_{I_s}\wedge\nu\big\}
 \subset\Om^{n,s}(\J^\infty E). \]
There is an induced ascending filtration on
$\Om^{n,s}_0=\Om^{n,s}_0(\J^\infty E)$ given by, for all $l\ge0$,
\[  \F^l\Om^{n,s}_0:=\Om^{n,s}_0\cap\F^l\Om^{n,s}=\Span_{\cR^\infty}\big\{\,
 \tht^{\al_1}\wedge\tht^{\al_2}_{I_2}\wedge\cdots\wedge\tht^{\al_s}_{I_s}
 \wedge\nu:|I_2|+\cdots+|I_s|\le l\,\big\}.  \]
The associated graded components are
\[ \Gr^l\Om^{n,s}_0:=\mfrac{\F^l\Om^{n,s}_0}{\F^{l-1}\Om^{n,s}_0}
   =\Span_{\cR^\infty}\big\{\,
    \tht^{\al_1}\wedge\cdots\wedge\tht^{\al_s}_{I_s}\wedge\nu
    :|I_2|+\cdots+|I_s|=l\,\big\}. \]
Again, both $\F^l\Om^{n,s}_0$ and $\Gr^l\Om^{n,s}_0$ are $\cR^\infty$-modules
for all $s\ge0$, $l\ge0$.

\section{The interior Euler operator, functional forms and cohomology}

With the decomposition of $p$-forms on $\J^\infty E$ into $(r,s)$-forms,
$r+s=p$, the exterior differentiation
$\dr\colon\Om^p(\J^\infty E)\to\Om^{p+1}(\J^\infty E)$ splits as
$\dr=\dr_H+\dr_V$, where
$\dr_H\colon\Om^{r,s}(\J^\infty E)\to\Om^{r+1,s}(\J^\infty E)$,
$\dr_V\colon\Om^{r,s}(\J^\infty E)\to\Om^{r,s+1}(\J^\infty E)$, are the
horizontal, vertical differentiations, respectively.
We have $\dr_H^2=0$, $\dr_V^2=0$, $\dr_H\dr_V+\dr_V\dr_H=0$.
In classical field theory, a Lagrangian is an element
$\lam\in\Om^{n,0}(\J^\infty E)$, so that the classical action is $\int_M\lam$,
assuming that either $M$ is compact or the fields decay sufficiently fast at
infinity.
The fundamental formula in the calculus of variation is
\[  \dr_V\lam=E_\al(\lam)\tht^\al+\dr_H\eta  \]
for some $\eta\in\Om^{n-1,1}(\J^\infty E)$.
Here $E_\al(\lam)$ is the Euler derivative of $\lam$, i.e.,
\[  E_\al(\lam):=\msum{|I|\ge0}{}(-D)_I\big(\pdr^I_\al\lam\big), \]
where $\pdr^I_\al:=\pdr/\pdr u^\al_I$ is a vector field on $\J^\infty E$,
$D_j:=\pdr_j+\sum_{k=0}^\infty u^\al_{Ij}\,\pdr^I_\al$ is the total derivative
along $x^j$, and $D_I:=D_{i_1}\cdots D_{i_k}$ for a multi-index
$I=(i_1,\dots,i_k)$.
The equations of motion are the Euler-Lagrange equations $E_\al(\lam)=0$.
Furthermore, $\eta$ can be expressed in terms of the higher Euler derivatives
$E_\al^I(\lam)$ of $\lam$, i.e.,
\[ \eta=\msum{|I|\ge0}{}D_I(E_\al^{Ij}(\lam)\tht^\al)
   \wedge(\pdr_j\lrcorner\,\nu),\qquad\qquad
E_\al^I(\lam):=\msum{J\supset I}{}\,
   \scalebox{.85}[.85]{$\begin{pmatrix}|J|\\ |I|\end{pmatrix}$}
   (-D)_{J\backslash I}(\pdr_\al^J\lam). \]

The above decomposition of $\dr_V\lam\in\Om^{n,1}(\J^\infty E)$ can be
generalised.
The interior Euler operator on $\om\in\Om^{r,s}(\J^\infty E)$~is
\[ I(\om):=\mfrac1s\,\tht^\al\wedge\msum{|I|\ge0}{}
   (-D)_I\big(\pdr^I_\al\,\lrcorner\,\om\big),  \]
which reduces to $E_\al\tht^\al$ when $\om=\dr_V\lam$.
Here $(r,s)$ is arbitrary, but with $r\ge0$, $s\ge1$, and $I(\om)$ remains in
$\Om^{r,s}(\J^\infty E)$.
Furthermore, any $\om\in\Om^{r,s}(\J^\infty E)$ decomposes uniquely as
$\om=I(\om)+\dr_H\eta$ for some $\eta\in\Om^{r-1,s}(\J^\infty E)$ that can be
obtained from $\om$ by higher Euler operators.
We have $I\circ\dr_H=0$ and $I^2=I$, i.e., $I$ is a projection operator or
an idempotent.
For all $s\ge1$, we set
$\cF^s(\J^\infty E):=I(\Om^{n,s}(\J^\infty E))$ or $\cF^s$ for short.
Elements of $\cF^s$ are called the functional $s$-forms.
For $s=1$, we have $\cF^1=\F^0\Om^{n,1}=\Gr^0\Om^{n,1}=\Om^{n,1}_0$, which is
an $\cR^\infty$-module of $\Om^{n,1}$.
But for general $s\ge2$, $\cF^s$ are $\bR$-linear subspaces of
$\Om^{n,s}_0(\J^\infty E)$ (see below).

The operators $\del_V\colon\cF^s\to\cF^{s+1}$ defined by $\del_V:=I\circ\dr_V$
for all $s\ge1$ satisfy $\del_V^2=0$, making $\cF^\bullet$ a cochain complex.
Making use of the Euler derivative operator $E\colon\Om^{n,0}\to\cF^1$, this
is part of the Euler-Lagrange complex
\[ 0\to\bR\to\Om^{0,0}\map{\dr_H}\Om^{1,0}\to\cdots\to\Om^{n-1,0}\map{\dr_H}
   \Om^{n,0}\map{E}\cF^1\map{\del_V}\cF^2\map{\del_V}\cF^3\to\cdots.  \] 
Its cohomology groups are
\[ H^p(\Om^{\bullet,0},\dr_H)\cong H^p(E),\;p\le n;\qquad
   H^s(\cF^\bullet,\del_V)\cong H^{n+s}(E),\;s\ge1. \]
On the other hand, for all $s\ge1$, we have an exact sequence
\[ 0\to\Om^{0,s}\map{\dr_H}\Om^{1,s}\to\cdots\to\Om^{n-1,s}\map{\dr_H}
   \Om^{n,s}\map{I}\cF^s\to0,  \]
known as the acyclicity of the variational bicomplex.
We refer the reader to \cite[Chap.~5]{An} for a comprehensive treatise.

\begin{proposition}\label{pro:IbarI}
1. The map $I$ restricted to $\Om^{n,s}_0$ preserves the filtration.
That is, $I(\F^l\Om^{n,s}_0)\subset\F^l\Om^{n,s}_0$ for all $l\ge0$.\\
2. The induced map $\bar I$ acting on the graded components $\Gr^l\Om^{n,s}_0$
are idempotent $\cR^\infty$-module homomorphisms.
\end{proposition}

\begin{proof}
1. We start with a typical element
\[ \eta_A:=A^{I_2\cdots I_s}_{\al_1\al_2\cdots\al_s}\tht^{\al_1}\wedge
  \tht^{\al_2}_{I_2}\wedge\cdots\wedge\tht^{\al_s}_{I_s}\wedge\nu
  \in\Gr^l\Om^{n,s}_0\subset\F^l\Om^{n,s}_0, \]
where the multi-indices $I_2,\dots,I_s$ are summed over but with
$|I_2|,\dots,|I_s|$ fixed so that $|I_2|+\cdots+|I_s|=l$, and
$A^{I_2\cdots I_s}_{\al_1\al_2\cdots\al_s}\in\cR^\infty$ are coefficients in
front of the $\cR^\infty$-module basis of $\Om^{n,s}_0$.
Its image under the map $I$ is
\[ I(\eta_A)=\mfrac1s\Big(A^{I_2\cdots I_s}_{\al_1\al_2\cdots\al_s}\tht^{\al_1}
\wedge\tht^{\al_2}_{I_2}\wedge\cdots\wedge\tht^{\al_s}_{I_s}+\,\msum{k=2}s\,
(-1)^{k-1}\tht^{\al_k}\wedge(-D)_{I_k}
\big(A^{I_2\cdots I_s}_{\al_1\al_2\cdots\al_s}
\tht^{\al_1}\wedge\tht^{\al_2}_{I_2}\wedge\cdots\widehat{\tht^{\al_k}_{I_k}}
\cdots\wedge\tht^{\al_s}_{I_s}\big)\Big)\wedge\nu. \]
Here and below, quantities under hats are omitted.
For each term in the sum on the right hand side, the derivatives in $D_{I_k}$ 
can act on either $A^{I_2\cdots I_s}_{\al_1\al_2\cdots\al_s}$ or on
$\tht^{\al_1}$, $\tht^{\al_j}_{I_j}$ ($j\ne k$) by the Leibniz rule.
The grading in the filtration is preserved if none of them acts on
$A^{I_2\cdots I_s}_{\al_1\al_2\cdots\al_s}$ and is decreased if otherwise.
Therefore $I(\F^l\Om^{n,s}_0)\subset\F^l\Om^{n,s}_0$.\\
2. The induced map $\bar I$ on the graded component $\Gr^l\Om^{n,s}_0$ is
defined as
$\bar I\colon\om+\F^{l-1}\Om^{n,s}_0\mapsto I(\om)+\F^{l-1}\Om^{n,s}_0$
for all $\om\in\F^l\Om^{n,s}_0$.
By Part~1, the map $\bar I$ is well defined, and it is an idempotent because
$I$ is.
By projecting $I(\eta_A)\in\F^l\Om^{n,s}_0$ calculated above to the graded
component $\Gr^l\Om^{n,s}_0$, we obtain
\[ \bar I(\eta_A)=\mfrac1sA^{I_2\cdots I_s}_{\al_1\al_2\cdots\al_s}
 \Big(\tht^{\al_1}\wedge\tht^{\al_2}_{I_2}\wedge\cdots\wedge\tht^{\al_s}_{I_s}
 +\,\msum{k=2}s(-1)^{k-1}\tht^{\al_k}\wedge(-D)_{I_k}\big(\tht^{\al_1}\wedge
  \tht^{\al_2}_{I_2}\wedge\cdots\widehat{\tht^{\al_2}_{I_2}}
  \cdots\wedge\tht^{\al_s}_{I_s}\big)\Big)\wedge\nu,  \]
which clearly shows that $\bar I$ is $\cR^\infty$-linear.
\end{proof}

Since the coefficients $A^{I_2\cdots I_s}_{\al_1\al_2\cdots\al_s}$ appears
inside the derivatives on the right hand side of the formula of $I(\eta_A)$,
the map $I$, although being $\bR$-linear, is not an $\cR^\infty$-module
homomorphism.
Consequently, neither the kernel $\ker(I)=\im(\dr_H)$ nor the image
$\cF^s=I(\Om^{n,s})$ are necessarily $\cR^\infty$-submodules of $\Om^{n,s}_0$.
A notable exception is $\cF^1$, which happens to be an $\cR^\infty$-module
because it is equal to the whole space $\Om^{n,1}_0$, as explained above.

\section{Filtration on the space of functional forms and graded components}

For the space of functional $s$-forms $\cF^s=\cF^s(\J^\infty E)
=I(\Om^{n,s}(\J^\infty E))\subset\Om^{n,s}_0(\J^\infty E)$, $s\ge1$, defined
above, we introduce an ascending filtration
\[ 0=\F^{-1}\,\cF^s\subset\F^0\,\cF^s\subset\F^1\,\cF^s\subset\cdots\subset
   \F^\infty\,\cF^s=\cF^s \]
from that of $\Om^{n,s}_0$, by taking
\[  \F^l\,\cF^s:=\cF^s\cap\F^l\Om^{n,s}_0. \]
Like $\cF^s\subset\Om^{n,s}_0$, the filtered spaces $\F^l\,\cF^s$ are
$\bR$-linear subspaces of $\F^l\Om^{n,s}_0$ but are not
$\cR^\infty$-submodules.
But we will show that the graded components
$\Gr^l\,\cF^s=\mfrac{\F^l\,\cF^s}{\F^{l-1}\,\cF^s}$ are isomorphic to
$\cR^\infty$-submodules of $\Gr^l\Om^{n,s}_0$ (Theorem~\ref{thm:GrF}.2).

\begin{lemma}\label{lem:GrF}
The graded components $\Gr^l\,\cF^s$, $l\ge0$, can be naturally identified as
$\bR$-linear subspaces of $\Gr^l\Om^{n,s}_0$.
\end{lemma}

\begin{proof}
Directly, the map $\Gr^l\,\cF^s\to\Gr^l\Om^{n,s}_0$ is
$\om+\F^{l-1}\,\cF^s\mapsto\om+\F^{l-1}\Om^{n,s}_0$ for all
$\om\in\F^l\,\cF^s$.
It is clearly well defined and is injective because if
$\om+\F^{l-1}\Om^{n,s}_0$ is zero in $\Gr^l\Om^{n,s}_0$, then
$\om\in\F^{l-1}\Om^{n,s}_0$ and thus
$\om\in\cF^s\cap\F^{l-1}\Om^{n,s}_0=\F^{l-1}\,\cF^s$, i.e.,
$\om+\F^{l-1}\,\cF^s$ is zero in $\Gr^l\,\cF^s$.
A more formal explanation applies the group homomorphism theorem.
We have
\[ \Gr^l\cF^s=\mfrac{\cF^s\cap\F^l\Om^{n,s}_0}{\cF^s\cap\F^{l-1}\Om^{n,s}_0}
  =\mfrac{\cF^s\cap\F^l\Om^{n,s}_0}
   {(\cF^s\cap\F^l\Om^{n,s}_0)\cap\F^{l-1}\Om^{n,s}_0}
  \cong\mfrac{(\cF^s\cap\F^l\Om^{n,s}_0)+\F^{l-1}\Om^{n,s}_0}
  {\F^{l-1}\Om^{n,s}_0}. \]
Since $(\cF^s\cap\F^l\Om^{n,s}_0)+\F^{l-1}\Om^{n,s}_0\subset\F^l\Om^{n,s}_0$,
$\Gr^l\,\cF^s$ is naturally isomorphic to a subspace of
$\mfrac{\F^l\Om^{n,s}_0}{\F^{l-1}\Om^{n,s}_0}=\Gr^l\Om^{n,s}_0$.
\end{proof}

\begin{lemma}\label{lem:FI}
For all $l\ge0$, we have $\F^l\,\cF^s=I(\F^l\Om^{n,s}_0)$.
\end{lemma}

\begin{proof}
First, by $I(\F^l\Om^{n,s}_0)\subset\F^l\Om^{n,s}_0$ and
$I(\F^l\Om^{n,s}_0)\subset I(\Om^{n,s}_0)=\cF^s$, we obtain
$I(\F^l\Om^{n,s}_0)\subset\F^l\Om^{n,s}_0\cap\cF^s=\F^l\,\cF^s$.
Conversely, if $\om\in\F^l\,\cF^s\subset\F^l\Om^{n,s}_0$, then since
$\om\in\cF^s$, we have $\om=I(\om)\in I(\F^l\Om^{n,s}_0)$.
\end{proof}

Lemma~\ref{lem:FI} implies that the filtration $\F^l\,\cF^s$ is controlled by
the restriction of $I$ to finite jet-order data.
The following is a refinement of Lemma~\ref{lem:GrF}, which equips
$\Gr^l\,\cF^s$ with $\cR^\infty$-module structures.

\begin{theorem}\label{thm:GrF}
$\Gr^l\,\cF^s$ are naturally $\bR$-linearly isomorphic to
$\Gr^l\,\cF^s\cong\bar I_l(\Gr^l\Om^{n,s}_0)$, which are $\cR^\infty$-modules.
\end{theorem}

\begin{proof}
For clarity, for each $l\ge0$, we denote by $I_l$ the restriction of $I$ to
$\F^l\Om^{n,s}_0$ and by $\bar I_l$ the induced operator on $\Gr^l\Om^{n,s}_0$.
Then under the splitting of $\cR^\infty$-modules
$\F^l\Om^{n,s}_0=\F^{l-1}\Om^{n,s}_0\oplus\Gr^l\Om^{n,s}_0$, $I_l$ assumes
the block-matrix form
\[ I_l=\begin{pmatrix} I_{l-1} & B_l\\ 0 & \bar I_l \end{pmatrix}  \]
for some $\bR$-linear operator
$B_l\colon\Gr^l\Om^{n,s}_0\to\F^{l-1}\Om^{n,s}_0$.
That is, for $\om\in\F^{l-1}\Om^{n,s}_0$ and $\eta\in\Gr^l\Om^{n,s}_0$ such
that $(\om,\eta)\in\F^l\Om^{n,s}_0$, we have
$I_l(\om,\eta)=(I_{l-1}\om+B_l\eta,\bar I_l\eta)$.
Typically, $B_l$ is not $\cR^\infty$-linear.
By using $I_l^2=I_l$, we get an identity $I_{l-1}B_l+B_l\bar I_l=B_l$ of linear
operators, which is equivalent to either of the following two identities:
\[ B_l\,(\id-\bar I_l)=I_{l-1}B_l,\qquad\qquad\qquad
   B_l\,\bar I_l=(\id-I_{l-1})\,B_l. \]
Since $\bar I_l$ is a projection operator on $\Gr^l\Om^{n,s}_0$ and is
$\cR^\infty$-linear (Proposition~\ref{pro:IbarI}.2), we have an
$\cR^\infty$-module decomposition
\[ \Gr^l\Om^{n,s}_0=\im(\id-\bar I_l)\oplus\im(\bar I_l).  \]
If $\om\in\im(\id-\bar I_l)$, then $\om=(\id-\bar I_l)\,\om$ and
$B_l\,\om=B_l\,(\id-\bar I_l)\,\om=I_{l-1}\,B_l\,\om\in\im(I_{l-1})$.
That is, we get $B_l(\im(\id-\bar I_l))\subset\im(I_{l-1})$.
On the other hand, if $\om\in\im(\bar I_l)$, then $\om=\bar I_l\,\om$ and
$B_l\,\om=B_l\,\bar I_l\,\om=(\id-I_{l-1})\,B_l\,\om\in\im(\id-I_{l-1})$.
Therefore $B_l(\im(\bar I_l))\cap\im(I_{l-1})=0$ and hence
$\Gam_l\cap\im(I_{l-1})=0$, where
$\Gam_l=\{(B_l\,\om,\om):\om\in\im(\bar I_l)\}$ is the graph of the operator
$B_l$ restricted to $\im(\bar I_l)$.
Consequently, by Lemma~\ref{lem:FI}, we obtain
\[ \F^l\,\cF^s=I_l(\F^l\Om^{n,s}_0)
   =I_{l-1}(\F^{l-1}\Om^{n,s}_0)+B_l(\im(\id-\bar I_l))+\Gam_l
   =\F^{l-1}\,\cF^s\oplus\Gam_l  \]
and the desired isomorphism (as $\bR$-vector spaces)
\[ \Gr^l\cF^s=\mfrac{\F^l\,\cF^s}{\F^{l-1}\,\cF^s}
   \cong\Gam_l\cong\im(\bar I_l), \]
where the last isomorphism is the projection
$\Gam_l\subset\im(\id-I_{l-1})\oplus\im(\bar I_l)\to\im(\bar I_l)$.
\end{proof}

\begin{corollary}\label{cor:GrF}
Fix $l\ge0$ and for each $k\ge l$, let $\Gam_k$ be the graph of $B_k$ on
$\im(\bar I_k)$ in the proof of Theorem~\ref{thm:GrF}. 
Then
\[ \F^l\,\cF^s=\Gam_l\oplus\Gam_{l-1}\oplus\cdots\oplus\Gam_1\oplus\Gam_0 \]
as $\bR$-vector spaces.
Consequently, each $\om\in\F^l\,\cF^s$ can be uniquely written as
$\om=\om_l+\om_{l-1}+\cdots+\om_1+\om_0$, where
$\om_k=(B_k\eta_k,\eta_k)\in\Gam_k$, $\eta_k\in\bar I_l(\Gr^l\Om^{n,s}_0)$,
$0\le k\le l$.
The condition $\om=0$ is equivalent to $\om_k=0$ for all $0\le k\le l$
as well as to $\eta_k=0$ for all $0\le k\le l$.
\end{corollary}

\begin{proof}
In the proof of Theorem~\ref{thm:GrF}, we have
$\F^l\,\cF^s=\F^{l-1}\,\cF^s\oplus\Gam_l$.
The decomposition
$\,\F^l\,\cF^s=\Gam_l\oplus\Gam_{l-1}\oplus\cdots\oplus\Gam_1\oplus\Gam_0$
follows inductively, and the rest is straightforward.
\end{proof}

When $s=2$, the decomposition of $\om$ as a sum of $\om_l,\dots,\om_1,\om_0$
was obtained in \cite[Proposition~3.6]{An}.
The result of Corollary~\ref{cor:GrF} holds for general $s\ge1$.
Furthermore, the advantage of $\eta_k=0$ over $\om_l=0$ is that $\eta_k$,
$0\le k\le l$, are in $\cR^\infty$-modules, and the vanishing of $\eta_k$
is conveniently achieved by setting to zero of coefficients of an
$\cR^\infty$-module basis.

\section{Two examples: the case $s=2$, $l\ge0$ arbitrary and the case
$l=1$, $s\ge0$ arbitrary}

For $s=2$, we have
\[ \Gr^l\Om^{n,2}_0
   =\Span_{\cR^\infty}\{\tht^\al\wedge\tht_I^\beta\wedge\nu:|I|=l\}. \]
An element in $\Gr^l\Om^{n,2}_0$ is of the form
$A^I_{\al\beta}\tht^\al\wedge\tht^\beta_I\wedge\nu$ where the multi-index $I$
with $|I|=l$ is summed.

\begin{corollary}\label{cor:s=2}
The induced map $\bar I_l:\Gr^l\Om^{n,2}_0\to\Gr^l\Om^{n,2}_0$ acts by
\[ A^I_{\al\beta}\tht^\al\wedge\tht_I^\beta\wedge\nu\longmapsto
   \mfrac12\big(A^I_{\al\beta}+(-1)^{l+1}A^I_{\beta\al}\big)
   \,\tht^\al\wedge\tht_I^\beta\wedge\nu. \]
Consequently,
\[  \im(\bar I_l)=\big\{A^I_{\al\beta}\tht^\al\wedge\tht_I^\beta\wedge\nu
    :A^I_{\al\beta}=(-1)^{l+1}A^I_{\beta\al}\big\},\qquad
    \ker(\bar I_l)=\big\{A^I_{\al\beta}\tht^\al\wedge\tht_I^\beta\wedge\nu
    :A^I_{\al\beta}=(-1)^lA^I_{\beta\al}\big\}.     \]
\end{corollary}

\begin{proof}
The map is the special case when $s=2$ of the formula for $\bar I$ in the
proof of Proposition~\ref{pro:IbarI}.2.
The results on its image and kernel are straightforward.
\end{proof}

Although well known, we use Corollary~\ref{cor:s=2} to obtain a necessary
condition that given a set of $\Del_\al\in\Om^{n,0}$, the equations
$\Del_\al=0$ are the Euler-Lagrange equations of some Lagrangian
$\lam\in\Om^{n,0}$.
For this to happen, $\Del:=\tht^\al\wedge\Del_\al$ is $\del_V\lam$ and must
satisfy $\del_V\Del=0$ in $\cF^2$.
We assume that all $\Del_\al\in\F^l\Om^{n,0}$ for some $l\ge1$.
We calculate
\[ \del_V\Del=I(\dr_V\Del)=\mfrac12\,\tht^\al\wedge\msum{0\le|I|\le l}{}\,
   \Big(\tht^\beta_I\wedge\pdr_\beta^I\Del_\al-(-D)_I(\tht^\beta\wedge
   \pdr_\al^I\Del_\beta)\Big). \]
The vanishing of $\del_V\Del$ is equivalent to that of all its graded
components, or $\mu_k=0$ for all $0\le k\le l$ as in Corollary~\ref{cor:GrF}.
For $k=l$, the only contributions to $\mu_l$ are from $I$ with $|I|=l$ in
the sum, and we have simply
\[  \mu_l=\mfrac12\,\msum{|I|=l}{}\,\tht^\al\wedge\tht^\beta_I\wedge
    (\pdr_\beta^I\Del_\al+(-1)^l\pdr_\al^I\Del_\beta)\in\Gr^l\Om^{n,2}.   \]
For $k=l-1$, there are contributions to $\mu_{l-1}$ from $|I|=l,l-1$ in the
sum.
For general $k$, the contributions to $\mu_k$ are from $|I|=l,l-1,\dots,l-k$.
We can rewrite the second term in the sum using higher Euler derivatives so
that
\[  \del_V\Del=\mfrac12\,\tht^\al\wedge\msum{0\le|I|\le l}{}\Big(\tht^\beta_I
    \wedge\big(\pdr_\beta^I\Del_\al-(-1)^{|I|}E_\al^I(\Del_\beta)\big)\Big).\]
Then for each $0\le k\le l$, we have
\[  \mu_k=\mfrac12\,\tht^\al\wedge\msum{|I|=k}{}\Big(\tht^\beta_I\wedge\big(
    \pdr_\beta^I\Del_\al-(-1)^{|I|}E_\al^I(\Del_\beta)\big)\Big)
    \in\Gr^k\Om^{n,2}. \]
Since $\tht^\al\wedge\tht^\beta_I\wedge\nu$, $|I|=k$, form a basis of the
$\cR^\infty$-module $\Gr^k\Om^{n,2}$, the vanishing of $\mu_k$ gives the
Helmholtz condition
\[   \pdr_\beta^I\Del_\al=(-1)^kE_\al^I(\Del_\beta),\qquad 0\le k\le l, \]
originally studied for the case $l=2$ \cite{He}.

For $s\ge3$, the map $\bar I_l$ in the proof of Proposition~\ref{pro:IbarI} is
quite explicit but involves sophisticated combinatorics if $l\ge2$.
We consider the special case $l=1$ and arbitrary $s\ge1$ as an illustration
of what the general result can be.

\begin{proposition}
Write a general element of $\Gr^1\Om^{n,s}\subset\F^l\Om^{n,s}$ as
\[  \eta_A=A^i_{\al_1\cdots\al_{s-1}\al_s}\tht^{\al_1}\wedge\cdots\wedge
    \tht^{\al_{s-1}}\wedge\tht^{\al_s}_i\wedge\nu, \]
where the coefficients $A^i_{\al_1\cdots\al_{s-1}\al_s}\in\cR^\infty$ are
totally antisymmetric in $\al_1,\dots,\al_{s-1}$.
Then
\begin{gather*}
\im(\bar I_1)=\Big\{\eta_A:\msum{k=0}{s-1}\,
 (-1)^{(s-1)k}A^i_{\al_{k+1}\cdots\al_s\al_1\cdots\al_{k-1}\al_k}=0\Big\}, \\
\ker(\bar I_1)=\Big\{\eta_A:sA^i_{\al_1\cdots\al_{s-1}\al_s}=\msum{k=0}{s-1}
 \,(-1)^{(s-1)k}A^i_{\al_{k+1}\cdots\al_s\al_1\cdots\al_{k-1}\al_k}\Big\}.
\end{gather*}
\end{proposition}

\begin{proof}
Using the definition of $I$, we calculate
\[ I_1(\eta_A)
   =\mfrac{s-1}s\eta_A-\mfrac1sD_i\big(A^i_{\al_1\cdots\al_{s-1}\al_s}
   \tht^{\al_1}\wedge\cdots\wedge\tht^{\al_{s-1}}\big)\wedge\tht^{\al_s}
   \wedge\nu,  \]
\begin{align*}
\bar I_1(\eta_A)&=
\Big(\mfrac{s-1}sA^i_{\al_1\cdots\al_{s-1}\al_s}+\mfrac1s\,\msum{k=1}{s-1}\,
 A^i_{\al_1\cdots\al_{k-1}\al_s\al_{k+1}\cdots\al_{s-1}\al_k}\Big)\,
 \tht^{\al_1}\wedge\cdots\wedge\tht^{\al_{s-1}}\wedge\tht^{\al_s}_i\wedge\nu \\
&=\Big(A^i_{\al_1\cdots\al_{s-1}\al_s}-\mfrac1s\,\msum{k=0}{s-1}\,
 (-1)^{(s-1)k}A^i_{\al_{k+1}\cdots\al_s\al_1\cdots\al_{k-1}\al_k}\Big)\,
 \tht^{\al_1}\wedge\cdots\wedge\tht^{\al_{s-1}}\wedge\tht^{\al_s}_i\wedge\nu.
\end{align*}
Here we used the total antisymmetry of $A^i_{\al_1\cdots\al_{s-1}\al_s}$
in the first $s-1$ indices to rewrite each term so that the indices are
cyclic permutations of $\al_1,\dots,\al_{s-1},\al_s$.
It is instructive to verify that $\bar I_1$ calculated here is indeed a
projection operator.
Put two column $s$-vectors: a unit vector
$c=(c_0,c_1,\dots,c_{s-1})^\mathrm{T}$, where $c_k=(-1)^{(s-1)k}/\sqrt s$, and
\[ a=(A^i_{\al_1\cdots\al_{s-1}\al_s},A^i_{\al_2\cdots\al_s\al_1},\dots,
     A^i_{\al_s\al_1\cdots\al_{s-1}})^\mathrm{T}.  \]
Then the above formula for $\bar I_1$ can be written as
$\bar I_1(a)=(I_s-cc^\mathrm{T})\,a$, where $I_s-cc^\mathrm{T}$ is clearly a
projection onto the hyperplane perpendicular to $c$.
The results on the image and kernel of $\bar I_1$ then follow.
\end{proof}

For small $s$, when $s=1$, $\bar I_1=0$.
When $s=2$, we recover the result of Corollary~\ref{cor:s=2} when $l=1$.
The first new case is when $s=3$.
For
$\eta_A=A^i_{\al\beta\gam}\tht^\al\wedge\tht^\beta\wedge\tht^\gam_i\wedge\nu$,
where $A^i_{\al\beta\gam}=-A^i_{\beta\al\gam}\in\cR^\infty$, we have
\[ \im(\bar I_1)=\{\eta_A:A^i_{\al\beta\gam}+A^i_{\gam\al\beta}
   +A^i_{\beta\gam\al}=0\},\qquad\qquad
   \ker(\bar I_1)=\{\eta_A:2A^i_{\al\beta\gam}=A^i_{\gam\al\beta}
   +A^i_{\beta\gam\al}\}. \]

\section{Conclusions and outlook}

In this work, we defined a filtration on the variational bicomplex according
to jet order, which is preserved by the interior Euler operator.
Furthermore, the space of functional forms of any fixed degree inherits
a filtration.
Though the filtered subspaces are not submodules, the graded components are
isomorphic to linear spaces which do have module structures.
In this way, the condition that a functional form vanishes can be stated
concisely using a module basis.
In a continuation of this work, we will consider cases with group symmetries
as well as the inverse of Noether's theorems using this new filtration.

\bigskip\noindent {\sc Acknowledgement.}
The work of S.W. is supported in part by the NSTC (Taiwan) grant
No.\;114-2115-M-007-003 and by the funding of the Mathematical Sciences
Institute Distinguished Research Visitor Program (MSRVP) of the Australian
National University.
He thanks the ANU, members of the MSI, and especially P.~Bouwknegt, for
hospitality.\\

\bibliographystyle{plain}

\end{document}